\documentclass[12pt]{amsart}

\allowdisplaybreaks

\usepackage{etex}
\usepackage[english]{babel}
\usepackage[cp1250]{inputenc}
\usepackage{amssymb,amsthm,amsfonts,amsmath,mathdots}
\usepackage{mathrsfs}
\usepackage{verbatim}
\usepackage[shortlabels]{enumitem}
\usepackage{titletoc}
\usepackage[title]{appendix}

\usepackage{url}
\usepackage{amsopn}
\usepackage{graphicx}
\usepackage[usenames, dvipsnames]{color}
\DeclareGraphicsExtensions{.pdf,.png,.jpg}

\usepackage{kbordermatrix}
\renewcommand{\kbldelim}{(}
\renewcommand{\kbrdelim}{)}
\renewcommand{\kbrowstyle}{\displaystyle}
\renewcommand{\kbcolstyle}{\displaystyle}

\usepackage[usenames, dvipsnames]{xcolor}
\DeclareGraphicsExtensions{.pdf,.png,.jpg}

\usepackage{xspace}

\DeclareMathOperator{\QM}{QM}

\DeclareMathOperator{\diag}{diag}

\DeclareMathOperator{\Rank}{rank}

\DeclareMathOperator{\Pos}{Pos}
\DeclareMathOperator{\fa}{fa}
\DeclareMathOperator{\cC}{im}
\newcommand{\Sym}{\mathbb{S}}

\makeatletter
\def\widebreve{\mathpalette\wide@breve}
\def\wide@breve#1#2{\sbox\z@{$#1#2$}%
     \mathop{\vbox{\m@th\ialign{##\crcr
\kern0.08em\brevefill#1{0.8\wd\z@}\crcr\noalign{\nointerlineskip}%
                    $\hss#1#2\hss$\crcr}}}\limits}
\def\brevefill#1#2{$\m@th\sbox\tw@{$#1($}%
  \hss\resizebox{#2}{\wd\tw@}{\rotatebox[origin=c]{90}{\upshape(}}\hss$}
\makeatletter

\usepackage[a4paper,margin=2.5cm]{geometry}
\newcommand{\x}{x}
\newcommand{\RR}{\mathbb R}

\newcommand{\NN}{\mathbb N}
\newcommand{\ZZ}{\mathbb Z}

\newcommand{\cM}{\mathcal M}

\newcommand{\cH}{\mathcal H}

\newcommand{\benu}{\begin{enumerate}}
\newcommand{\eenu}{\end{enumerate}}
\newcommand{\bop}{\begin{opomba}}
\newcommand{\eop}{\end{opomba}}

\newcommand{\Bor}{\mathrm{Bor}}

\newcommand{\tr}{\mathrm{tr}}

\newtheorem{theorem}{Theorem}[section]

\newtheorem{lemma}[theorem]{Lemma}
\newtheorem{proposition}[theorem]{Proposition}

\newtheorem*{problem*}{Problem}

\theoremstyle{definition}

\newcommand{\mbf}{\mathbf}

\theoremstyle{remark}
\newtheorem{remark}[theorem]{Remark}

\numberwithin{equation}{section}

\makeatletter
\newcommand{\pushright}[1]{\ifmeasuring@#1\else\omit\hfill$\displaystyle#1$\fi\ignorespaces}
\newcommand{\pushleft}[1]{\ifmeasuring@#1\else\omit$\displaystyle#1$\hfill\fi\ignorespaces}
\makeatother

\newcommand{\CB}[1]{\begin{color}{black}#1\end{color}}

\author[S. Sun]{Shengding Sun${}^{1}$}
\address{Shengding Sun,
Department of Applied Mathematics and Theoretical Physics, University of Cambridge, Cambridge, United Kingdom.}
\email{ss3104@cam.ac.uk, s.sun28@lse.ac.uk}
\thanks{${}^1$Supported by UKRI Horizon Europe EP/X032051/1.}

\author[A. Zalar]{Alja\v z Zalar${}^{2}$}
\address{Alja\v z Zalar, 
Faculty of Computer and Information Science, University of Ljubljana  \& 
Faculty of Mathematics and Physics, University of Ljubljana  \&
Institute of Mathematics, Physics and Mechanics, Ljubljana, Slovenia.}
\email{aljaz.zalar@fri.uni-lj.si}
\thanks{${}^2$Supported by the ARIS (Slovenian Research and Innovation Agency)
research core funding No.\ P1-0288 and grants No.\ J1-50002, J1-6011.}


\date{\today}

\begin{document}

\numberwithin{equation}{section}

\dottedcontents{section}[3.8em]{}{2.3em}{.4pc} 
\dottedcontents{subsection}[6.1em]{}{3.2em}{.4pc}
\dottedcontents{subsubsection}[8.4em]{}{4.1em}{.4pc}

\begin{abstract}
The matrix Fej\'er-Riesz theorem characterizes positive semidefinite matrix
	polynomials on the real line.
 In \cite{Zal16} this was extended to the characterization on arbitrary closed semialgebraic sets $K\subseteq \RR$ by
	using matrix quadratic modules from real algebraic geometry.
	In the compact case there is a denominator-free characterization, while in the non-compact case denominators are needed except when $K$ is the whole line, an unbounded interval, a union of two unbounded intervals, and according to a conjecture of \cite{Zal16} also when $K$ is a union of an unbounded interval and a point or a union of two unbounded intervals and a point. In this paper, we confirm this conjecture by solving the truncated matrix-valued moment problem on a union of a bounded interval and a point.
 The presented technique for solving the corresponding \CB{moment problem} can potentially be used to determine degree bounds in the positivity certificates for matrix polynomials on compact sets $K$
 \cite[Theorem C]{Zal16}.
\end{abstract}

\title{Matrix Fej\'er-Riesz type theorem for a union of an interval and a point}

\maketitle


\section{Introduction}

\CB{
The classic Fej\'er-Riesz theorem \cite{Fej}
states that any nonnegative real-valued trigonometric polynomial is a single Hermitian square of a holomorphic trigonometric polynomial of the same degree. The matrix Fej\'er-Riesz theorem is its generalization to matrix-valued polynomials. Its real  version characterizes positive semidefinite matrix polynomials on the real line: 

\begin{theorem} 
\label{thm:Fejer-Riesz}
Let $p\in \mathbb{N}$ and let
$
F(x)=\sum_{i=0}^{2n}F_i x^i
$
be a $p\times p$ matrix polynomial whose coefficients $F_i$ are real symmetric matrices, with $F_{2n}\neq 0$.  
Assume that $F(x)$ is positive semidefinite for every $x\in \mathbb{R}$. Then there exist $p\times p$ matrix polynomials  
\[
H(x)=\sum_{i=0}^{n}H_i x^i
\quad\text{and}\quad
G(x)=\sum_{i=0}^{n}G_i x^i,
\]
whose coefficients $H_i,G_i$ are real symmetric matrices, such that  
\begin{equation}
\label{thm:Fejer-Riesz-eq}
F(x)=H(x)^T H(x)+G(x)^T G(x),
\end{equation}
where
\(
H(x)^T=\sum_{i=0}^{n}H_i^T x^i
\)
and
\(
G(x)^T=\sum_{i=0}^{n}G_i^T x^i.
\)
\end{theorem}

For various proofs of this result, we refer the reader to  
\cite{Gohberg-Krein, Jakubovic, Popov, Rosenblatt, Rosenblum}.  
Among these, the formulation in \cite[Appendix B]{Popov} is the most informative, as it shows the following: for every factorization of $\det F(x)$ of the form  
\[
\det F(x)=f(x)^\ast f(x),
\]
where $f(x)=\sum_{i=0}^{np} f_i x^i\in \mathbb{C}[x]$ and 
\(f(x)^\ast=\sum_{i=0}^{np}\overline{f_i}\,x^i\), there exist polynomials $H(x)$ and $G(x)$ satisfying \eqref{thm:Fejer-Riesz-eq} and  
\[
\det\!\bigl(H(x)+\mathfrak{i}\, G(x)\bigr)=f(x),
\]
where $\mathfrak{i}$ denotes the imaginary unit.
}

In \cite{Zal16} the second named author has extended a characterization to arbitrary closed semialgebraic sets $K\subseteq \RR$ by using matrix quadratic modules from real algebraic geometry. Except for two special cases, the following characterizations were obtained: 

\begin{theorem}[Theorem C and D in \cite{Zal16}]\label{thm:Zal16}
    Let $p$ be the size of matrix and $K\subseteq \RR$ be a closed semialgebraic set. 
    \begin{enumerate}
        \item If $K$ is compact, then the matrix quadratic module is saturated for every $p\ge 1$ and \emph{saturated} description of $K$. 
        \item If $K=\RR$ or $K$ is a union of two unbounded intervals, then the matrix quadratic module is saturated for every $p\ge 1$ and the \emph{natural} description of $K$. 
        \item If $K$ is not compact, and is not
        \begin{enumerate}[(i)]
            \item $\RR$ or a union of two unbounded intervals,
            \item\label{conj:case-1} a union of an unbounded interval and a point, or
            \item\label{conj:case-2} a union of two unbounded intervals and a point,
        \end{enumerate}
        Then the matrix preordering is NOT saturated for every $p\ge 2$ and any valid description of $K$. 
    \end{enumerate}
\end{theorem}

The main technique used in \cite{Zal16} to establish the characterizations is induction on the size of the matrices, where the base case is the scalar result of Kuhlmann and Marshall \cite[Theorem 2.2]{K-M}, the diagonalization method for matrix polynomials of Schm\"udgen \cite[\S4.3]{Sch} and in the compact case the elimination of the denominator by a result of Scheiderer \cite[Proposition 2.7]{Scheiderer}. In the diagonalization method, the degrees of the polynomials can grow exponentially with the size of the matrices, while \cite[Proposition 2.7]{Scheiderer} uses the Stone-Weierstra{\ss} theorem, in which the trace of the degrees is lost.

The \emph{natural} description is a special case of saturated descriptions. Thus Theorem \ref{thm:Zal16} states that except for the two unresolved cases \ref{conj:case-1} and \ref{conj:case-2}, the natural description is the best possible in terms of saturation of matrix quadratic module and preordering. 
It was conjectured in \cite{Zal16} that for cases \ref{conj:case-1} and \ref{conj:case-2}, the matrix preordering is saturated for the natural description and every $p\ge 1$.

When the preordering is not saturated for any valid description, there is a weaker characterization with denominators \cite[Theorem D]{Zal16}. It turns out that by an appropriate substitution of variables, both unresolved cases \ref{conj:case-1} and \ref{conj:case-2} above are equivalent. Moreover, they are equivalent to the case of a union of a compact interval and a point, where the degrees in the algebraic certificate of positivity are bounded by the degree of the given positive semidefinite matrix polynomial. 
\CB{The details of these claims are presented in Section \ref{sec:psatze-one-unbounded-interval}; here, we only outline the basic idea.} After applying a suitable substitution of the
variables $\phi$, a union of an unbounded interval and a point $K$ becomes a union of a compact interval and a point $\phi(K)$, while the given matrix polynomial $F$, positive semidefinite on $K$, becomes a matrix polynomial $G=\phi(F)$, positive semidefinite on $\phi(K)$. By the form of the polynomials in the natural description of $K$, it is clear the bounds on the degrees of the summands in the representation of $F$ as an element of the matrix preordering are bounded by the degree of $F$ and hence the same is true for the corresponding representation of $G$ on $\phi(K)$. Conversely, any representation of $G$ as an element of the matrix preordering corresponding to the natural description of $\phi(K)$, yields a representation of $F$ as an element of the matrix preordering corresponding to the natural description of $K$ only if the degrees of summands are bounded by the degree of $G$. Otherwise non-trivial denominators would be present in the representation. 

In this article we answer the conjecture 
above affirmatively and conclude all univariate cases of matrix Fej\'er-Riesz theorem by studying the dual side of solving the corresponding univariate truncated matrix-valued moment problem on $K$.
\CB{For 
$K$ being an interval (bounded or unbounded)}, the problem is well understood and has been solved using tools from several fields, including operator theory, complex analysis and linear algebra \cite{And70,DS02,CDFK06,DFKM10,BW11}.
Very recently, the study of operator-valued and multivariate \CB{truncated matrix-valued moment problems} has attracted interest among several authors \cite{KW13,Kim14,MS16,Kim22,KT22,MS23+,MS23++,CEAZ25}.
In particular, by \cite[Corollary 5.2]{MS23++}, the
\CB{truncated matrix-valued moment problem} on a compact set $K$ admits a solution if only if the corresponding linear functional is positive on every matrix polynomial of bounded degree, positive semidefinite on $K$.
Usually, this duality is exploited in the scalar case in the direction $(\Leftarrow)$ using certificates of positivity for polynomials. However, our motivation is to use the implication $(\Rightarrow)$ and obtain matricial sum-of-squares certificates by solving the \CB{truncated matrix-valued moment problem} for a given $K$.
\medskip

\subsection{Main results and reader's guide}

Given a $p\times p$ univariate matrix polynomial $H(\x)=\sum_{i=0}^n H_i\x^i$,
where $H_i$ are real matrices,
we call an expression of the form $H(\x)^TH(\x)$
a \textbf{symmetric square},
where $H(\x)^T=\sum_{i=0}^{n} H_i^T \x^i$. An algebraic certificate of positivity for matrix polynomials on the union of a bounded interval and a point is indeed the best possible in terms of degree bounds, which solves the conjecture from \cite{Zal16}:

\begin{theorem}
\label{thm:bounded-interval}
Let $K=\{a\}\cup [b,c]$, $a,b,c\in \RR$, $a<b<c$, and
    $F(\x)=\sum_{i=0}^n F_i \x^i$
    be a $p\times p$ matrix polynomial with coefficients $F_i$ being real symmetric matrices and $F_n$ a nonzero matrix.
Assume that $F(x)$ is positive semidefinite for every $x\in K$.
 Then there are 
 $p\times p$ real matrix polynomials $G_0(\x),G_1(\x),G_2(\x),G_3(\x)$,
 each being a sum of at most two symmetric squares,
    such that:
\smallskip
\begin{enumerate}[leftmargin=0.7cm]
\item If $n$ is even, then
        $$
        F(\x)=G_0(\x)+(\x-a)(\x-b)G_1(\x)+(\x-a)(c-\x)G_2(\x),
        $$
        where the degree of each summand is bounded above by the degree of $F$.
\smallskip
\item If $n$ is odd, then
\begin{align*} 
    F(\x)&=(\x-a)G_0(\x)+(c-\x)G_1(\x)+(\x-a)^2(\x-b)G_2(\x)\\
    &\hspace{5cm}+(\x-a)(\x-b)(c-\x)G_3(\x),
\end{align*}
where the degree of each summand is bounded above by the degree of $F$.
\end{enumerate}
\end{theorem}

Theorem \ref{thm:bounded-interval} implies the following positivity certificate for $K$ being a union of a point and an unbounded interval, which in particular proves that the matrix preordering generated by the natural description of $K$ is saturated.

\begin{theorem}
\label{thm:unbounded-interval}
Let $K=\{a\}\cup [b,\infty)$, $a,b\in \RR$, $a<b$, and
    $F(\x)=\sum_{i=0}^{n}F_i\x^i$
be a $p\times p$ matrix polynomial with coefficients $F_i$ being real symmetric matrices and $F_n$ a nonzero matrix.
Assume that $F(x)$ is positive semidefinite for every $x\in K$.
 Then there exist $p\times p$ real matrix polynomials  $G_0(\x),G_1(\x),G_2(\x), G_3(\x)$,  each being a sum of at most two symmetric squares,
    such that
        $$F(\x)=
        G_0(\x)+(\x-a)G_1(\x)+(\x-a)(\x-b)G_2(\x)+
        (\x-a)^2(\x-b)G_3(\x),
        $$
    where the degree of each summand is bounded above by the degree of $F$.
\end{theorem}

When $K$ is a union of a point and two unbounded intervals the positivity certificate is the following, implying that the matrix preordering generated by the natural description of $K$ is saturated.

\begin{theorem}
\label{thm:two-intervals}
Let $K=(-\infty,a]\cup\{b\} \cup [c,\infty)$, $a,b,c\in \RR$, $a<b<c$, and
    $F(\x)=\sum_{i=0}^{n}F_i\x^i$
be a $p\times p$ matrix polynomial with coefficients $F_i$ being real symmetric matrices and $F_n$ a nonzero matrix.
Assume that $F(x)$ is positive semidefinite for every $x\in K$.
 Then $n$ is even and there exist $p\times p$ real matrix polynomials  $G_0(\x),G_1(\x),G_2(\x), G_3(\x)$,  each being a sum of at most two symmetric squares,
    such that
    \begin{align*} 
        F(\x)&=
        G_0(\x)+(\x-a)(\x-b)G_1(\x)+(\x-b)(\x-c)G_2(\x)+\\
        &\hspace{5cm}+(\x-a)(\x-b)^2(\x-c)G_3(\x),
    \end{align*}
    where the degree of each summand is bounded above by the degree of $F$.
\end{theorem}

The paper is organized as follows. In Section \ref{sec:notation}
we introduce the notation and some preliminary results. We state the truncated matrix-valued moment problem and establish the connection with positive matrix polynomials
(see Proposition \ref{prop:140525}).
In Section \ref{sec:K-flat-extension} we extend the flat extension theorem for a closed semialgebraic set $K$ in $\RR$ from the scalar to the matrix case
(see Theorem \ref{main-theorem-060525}). In Section \ref{sec:coflatness}
we prove the main technical result, Proposition \ref{lem:aux-3},
which allows us to manipulate the value of the matrix moment of degree 0, and is used in Section \ref{sec:union-interval-point} to solve the truncated matrix-valued moment problem on the union of a compact interval and a point (see Theorem \ref{main-theorem-1}).
This result together with Proposition \ref{prop:140525} then implies Theorem \ref{thm:bounded-interval}.
\CB{Finally, in Section \ref{sec:psatze-one-unbounded-interval}, we prove Theorems \ref{thm:unbounded-interval} and \ref{thm:two-intervals}
using Theorem \ref{thm:bounded-interval} and elementary algebraic manipulations.}

We mention at the end that the approach presented in this paper could be used to provide degree bounds in the positivity certificates for matrix polynomials on a compact set $K$
(see \cite[Theorem C]{Zal16}),
if Theorem \ref{main-theorem-1} can be extended to a given $K$.


\section{Notation and preliminaries}
\label{sec:notation}

Let $d\in \NN\cup\{0\}$, $p\in \NN$. We denote by $M_p(\RR)$ the set of $p\times p$ real matrices and by $\Sym_p$ its subset of symmetric matrices.
For $A\in \Sym_p$ the notation $A\succeq 0$ (resp.\ $A\succ 0$) means $A$ is positive semidefinite (resp.\ positive definite).
We use $\Sym^{\succeq 0}_p$ for the subset  of all \CB{positive semidefinite} matrices in $\Sym_p$.
Let $\tr$ denote the trace and 
$\langle\cdot,\cdot\rangle$ the usual Frobenius inner product on $M_p(\RR)$, i.e., $\langle A, B\rangle =\tr(A^TB).$

Let $\RR[\x]_{\leq d}$ stand for the vector space of univariate polynomials of degree at most $d$. 
Let $M_p(\RR[\x])$ be a set of all $p \times p$ 
matrix polynomials over $\RR[\x]$.
We say $F(\x) \in M_p(\RR[\x])$ is \textbf{symmetric} if 
$F(\x) = F(\x)^T$. 
    We write $$\CB{\Sym_p[\x]_{\leq d}:=\Big\{\sum_{k=0}^{d}A_k\x^k\colon A_0,\ldots,A_d\in \Sym_p\Big\}}$$
 for the set of all symmetric matrix polynomials from $M_p(\RR[\x])$ \CB{of degree at most $d$}.  

\subsection{Positive matrix polynomials and matrix quadratic module}
A matrix polynomial $F(\x) \in \CB{\Sym_p[\x]}$ is positive semidefinite in 
$x_0 \in \RR$ if $v^TF(x_0)v \geq 0$ for
 every $v \in \RR^p$. 
 We denote by $\sum M_p(\RR[\x])^2$ the set of sums of symmetric squares $H(\x)^T H(\x)$, where  $H(\x)\in M_p(\RR[\x])$. Note that by the matrix Fej\'er-Riesz theorem \CB{(see Theorem \ref{thm:Fejer-Riesz} above)}, 
$\sum M_p(\RR[\x])^2$ is equal to the set of all sums of at most two symmetric squares.

Let $K\subseteq \RR$ be a closed nonempty set.
We denote by $\Pos_{d}^{(p)}(K)$
the set of $p\times p$ matrix polynomials over $\RR[\x]_{\leq d}$, \CB{positive semidefinite} on $K$, i.e., 
$$
    \Pos_{d}^{(p)}(K)
    :=
    \{
    F\in \Sym_p(\RR[\x]_{\leq d}) \colon   
        F(x)\succeq 0 \text{ for every }x\text{ in }K
    \}.
$$ 
We call every $F\in \Pos_{d}^{(p)}(K)$ a \textbf{$K$--positive matrix polynomial}.

Let $S\subset \RR[\x]$ be a finite set.
We denote by 
\begin{equation*}
    K_S:=\{x\in\RR\colon f(x)\geq 0\;\text{for each }f\in S\}
\end{equation*}
the semialgebraic set generated by $S$.
The \textbf{matrix quadratic module generated by $S$} in $M_p(\RR[\x])$ is defined by
	\begin{eqnarray*}
		\QM^{(p)}_{S} 
			&:=& 
				\Big\{\sum_{
                s\in S} \sigma_s s\colon \sigma_s\in 
				\sum M_p(\RR[\x])^2\;\text{for each}\;s\Big\}.
	\end{eqnarray*}
For $d\in \NN\cup\{0\}$ we define the set 
	\begin{align*} 
		\QM^{(p)}_{S,d}
            &:=
            \Big\{\sum_{s\in S} 
            \sigma_s s\colon 
			\sigma_s \in \sum M_p(\RR[\x])^2
            \text{ and }\deg(\sigma_s s)\leq d\; \text{for each}\;s\Big\}.
	\end{align*}
	We call $\QM^{(p)}_{S,d}$ the \textbf{$d$-th truncated matrix quadratic module generated by $S$}.


\begin{proposition}
    \label{prop:matrix-quadratic-closed} 
    If $K_S$ has a nonempty interior, then
    $\QM^{(p)}_{S,d}$ is closed in $\CB{\Sym_p[x]_{\leq d}}$.
\end{proposition}

\begin{proof}
The proof is analogous to the proof of the scalar result \cite[Theorem 3.49]{Lau10}, i.e., $p=1$.
\end{proof}

\smallskip
\subsection{Matrix measures}

Let $K\subseteq \RR$ be a closed set 
and $\Bor(K)$
the Borel $\sigma$-algebra of $K$. 
We call 
	$$\mu=(\mu_{ij})_{i,j=1}^p:\Bor(K)\to \Sym_p$$ 
a $p\times p$ \textbf{Borel matrix-valued measure} supported on $K$ (or \textbf{positive $\Sym_p$-valued measure} for short)
if:
\begin{enumerate}
	\item 
		$\mu_{ij}:\Bor(K)\to \RR$ 
			is a real measure for every $i,j=1,\ldots,p$ and
            \smallskip
	\item
    $\mu(\Delta)\succeq \CB{0}$ for every $\Delta\in \Bor(K)$.
\end{enumerate}
\smallskip

A point $x\in K$ is called an atom of $\mu$ if $\mu(\{x\})\neq \CB{0}.$
A $\Sym_p$-valued measure $\mu$ is \textbf{finitely atomic}, if there exists a finite set $M\in \Bor(K)$
such that
$\mu(K\setminus M)=\CB{0}$
or
equivalently,
$\mu=\sum_{j=1}^k A_j\delta_{x_j}$
for some $k\in \NN$, $x_j\in K$, $A_j\in \Sym_p^{\succeq 0}$,
where $\delta_{x_j}$
stands for the Dirac measure \CB{supported on $x_j\in \RR$}.

We denote by 
$\cM_p(K,\Bor(K))$
the set of all $\Sym_p$-valued measures
and by
$\cM^{(\fa)}_p(K,\Bor(K))$
the set of all finitely atomic $\Sym_p$-valued measures.

Let $\mu\in \cM_p(K,\Bor(K))$
and 
$\tau:=\tr(\mu)=\sum_{i=1}^p \mu_{ii}$ denote its trace measure. 
A polynomial $f\in \RR[x]_{\leq n}$ is $\mu$-integrable if 
$f\in L^1(\tau)$. 
We define its integral by
$$
\int_K f\;d\mu
=
\Big(\int_K f\; d\mu_{ij}\Big)_{i,j=1}^p.
$$

\smallskip

\subsection{Truncated matrix-valued moment problem}

Let $n,p\in \NN$.
Given a linear mapping
\begin{equation}
\label{linear-operator}
    L:\RR[x]_{\leq n}\to \Sym_p,
\end{equation}
the \textbf{truncated matrix-valued moment problem} supported on $K$
asks to characterize the existence of a $\Sym_p$-valued measure $\mu\in \cM(K,\Bor(K))$ 
such that
	\begin{equation}
		\label{moment-measure-cond}
			L(f)=\int_{K} f\; d\mu\quad \text{for every}\quad f\in \RR[x]_{\leq n}.
	\end{equation}
If such a measure exists, we say that $L$ 
is \textbf{a $K$--matrix moment functional on $\RR[x]_{\leq n}$}
and 
$\mu$ is its $K$--\textbf{representing matrix-valued measure.} We denote by $\cM_L$ the set of all $K$--\CB{representing matrix--valued measures} for $L$.

Equivalently, one can define $L$ as in \eqref{linear-operator}
by a sequence of its values on monomials $x^i$, $i=0,\ldots,n$. 
Throughout the paper we will denote these values
by $\Gamma_i:=L(x^i)$. 
If
\begin{equation} 
\label{moment-sequence}
    \Gamma:=(\Gamma_0,\Gamma_1,\ldots,\Gamma_n)\in (\Sym_p)^{n+1}
\end{equation}
is given, then we denote the corresponding linear mapping on 
$\RR[\x]_{\leq n}$ by $L_{\Gamma}$ and call it a \textbf{Riesz 
mapping of $\Gamma$}. 
If $L_\Gamma$ is a \CB{$K$--matrix moment functional}, we call $\Gamma$ a 
\textbf{$K$--matrix moment sequence}.\\

The connection between the \CB{truncated matrix-valued moment problem on $K$} and $K$--positive matrix polynomials
is the following.

\begin{proposition}
    \label{prop:140525}
    Let $n,p\in \NN$, $\Gamma$ as in \eqref{moment-sequence}
    and $K$ a compact set. The following statements are equivalent:
    \begin{enumerate}
    \item 
    $\Gamma$ is a $K$--matrix moment sequence.
    \smallskip
    \item 
        $\sum_{i=0}^n A_ix^i\in \Pos_n^{(p)}(K)$
        implies that
        $\sum_{i=0}^n \tr(\Gamma_i A_i)\geq 0.$ 
    \end{enumerate}
\end{proposition}

\begin{proof}
The proof is verbatim the same to the proof of \cite[Lemma 2.3(a)]{DS02} which deals with the case $K=[0,1]$.
Here we emphasize $K$ needs to be \textit{compact} for the proof to work, since the
set 
$$\cM(K,n,p):=\{\Gamma=(\Gamma_0,\ldots,\Gamma_n)\in(\Sym_p)^{n+1}\colon \Gamma\text{ is a }K\text{--matrix moment sequence}\}$$
needs to be closed. 
\end{proof}

\begin{remark}
\CB{In \cite[Section~2]{DS02} and \cite[Section~5]{DS02}, the authors study the truncated matrix-valued moment problem on $[0,1]$ and $[0,\infty)$, respectively. The proofs of the results in \cite[Section~2]{DS02} are given in full detail, whereas in \cite[Section~5]{DS02} the authors merely state that the corresponding results remain valid; in particular, this is claimed for Proposition~\ref{prop:140525}. However, this assertion is not correct, even in the scalar case ($p=1$). \CB{For example, the sequence $\Gamma=(0,\ldots,0,\Gamma_n=1)$ is not a $[0,\infty)$--moment sequence, even though
$
\sum_{i=0}^n \Gamma_i a_i = a_n \ge 0
$
for every polynomial $\sum_{i=0}^n a_i x^i$ that is nonnegative on $[0,\infty)$.}

}
\end{remark}

\subsection{Moment matrix and localizing moment matrices}

For $m,n\in \NN$, $m\leq \frac{n}{2}$ and $\Gamma$ as in 
\eqref{moment-sequence} we denote by 
$$
	\cM_m\equiv
    \cM_m(\Gamma)
	=\begin{pmatrix}\Gamma_{i+j-2}\end{pmatrix}_{i,j=1}^{m+1}
	=\begin{pmatrix}
			\Gamma_{0} & \Gamma_1 & \Gamma_2 &\cdots & \Gamma_m\\
			\Gamma_1 & \Gamma_2 & \iddots & \iddots & \Gamma_{m+1} \\
			\Gamma_2 & \iddots & \iddots & \iddots & \vdots \\
                \vdots & \iddots & \iddots & \iddots & \Gamma_{2m-1} \\
			\Gamma_m & \Gamma_{m+1} & \cdots & \Gamma_{2m-1} & \Gamma_{2m}
		\end{pmatrix}
$$
the \textbf{$m$--th truncated moment matrix.}
For $0\leq i,j\leq \frac{n}{2}$ we also write
\begin{equation}
    \label{def:columns-mm}
    \textbf{v}_i^{(j)}
	=\begin{pmatrix}\Gamma_{i+r-1}\end{pmatrix}_{r=1}^{j+1}
	=\begin{pmatrix}
		\Gamma_{i} \\ \Gamma_{i+1} \\ \vdots \\ \Gamma_{i+j}
 	\end{pmatrix}
\end{equation}

For $f\in \RR[\x]_{\leq n}$ an \textbf{$f$--localizing moment matrix $\cH_{f}$ of $L:\RR[\x]_{\leq n}\to \Sym_p$} is a block square matrix of size 
$s(n,f)\times s(n,f)$, where 
    $s(n,f)=\lfloor \frac{n-\deg f}{2}\rfloor +1$,
with the $(i,j)$--th entry equal to $L(fx^{i+j-2})$.
Writing $\Gamma^{(f)}_i:=L(fx^i)$, the 
\textbf{$\ell$-th  truncated $f$--localizing matrix} is
\begin{equation*}
    \label{localized-truncated}
		\cH_{f}(\ell):=\left(\Gamma^{(f)}_{i+j-2} \right)_{i,j=1}^{\ell+1}
					=	\kbordermatrix{
							& \textit{1} & X & X^2 & \cdots  & X^{\ell} \\[0.2em]
							\textit{1} & \Gamma^{(f)}_0 & \Gamma^{(f)}_1 & \Gamma^{(f)}_2 & \cdots & \Gamma^{(f)}_\ell\\[0.2em]
							X & \Gamma^{(f)}_1 & \Gamma^{(f)}_2 & \iddots & \iddots & \Gamma^{(f)}_{\ell+1}\\[0.2em]
							X^2 & \Gamma^{(f)}_2 & \iddots & \iddots & \iddots & \vdots\\[0.2em]
							\vdots &\vdots 	& \iddots & \iddots & \iddots & \Gamma^{(f)}_{2\ell-1}\\[0.2em]
							X^\ell & \Gamma^{(f)}_\ell & \Gamma^{(f)}_{\ell+1} & \cdots & \Gamma^{(f)}_{2\ell-1} & \Gamma^{(f)}_{2\ell}
						}
	\end{equation*}


For $0\leq i,j\leq s(n,f)$ we also write
\begin{equation}
    \label{def:columns-loc-mm}
    (f\cdot \textbf{v})_i^{(j)}
	:=\begin{pmatrix}\Gamma^{(f)}_{i+r-1}\end{pmatrix}_{r=1}^{j+1}
	=\begin{pmatrix}
		\Gamma_{i}^{(f)} \\ \Gamma^{(f)}_{i+1} \\ \vdots \\ \Gamma_{i+j}^{(f)}
 	\end{pmatrix}.
\end{equation}

\subsection{Generalized Schur complements}\label{SubS2.1}
Let $m,n\in \NN$ and 
	\begin{equation}\label{matrixM}
		\cM=\left( \begin{array}{cc} A & B \\ C & D \end{array}\right)\in M_{n+m}(\RR),
	\end{equation}
where $A\in \RR^{n\times n}$, $B\in \RR^{n\times m}$, $C\in \RR^{m\times n}$  and $D\in \RR^{m\times m}$.
The \textbf{generalized Schur complement} of $A$ (resp.\ $D$) in $\cM$ is defined by
	$$\cM/A=D-CA^\dagger B\quad(\text{resp.}\; \cM/D=A-BD^\dagger C),$$
where $A^\dagger $ (resp.\ $D^\dagger $) stands for the Moore-Penrose inverse of $A$ (resp.\ $D$)  \cite{Zha05}. 

\CB{For a matrix $M$ we denote by $\cC M$ its image and by $\ker M$ its kernel.}

A characterization of \CB{positive semidefinite} $2\times 2$ block matrices in terms of 
generalized Schur complements is the following.

\begin{theorem}[{\cite{Alb69}}] 
	\label{block-psd}   
	Let $n,m\in \NN$ and 
	$$
        \cM=
        \left( 
            \begin{array}{cc} 
                A & B \\ 
                B^T & C 
            \end{array}\right)
        \in M_{n+m}(\RR),$$
	where $A\in \Sym_n$, $B\in \RR^{n\times m}$,
        $C\in \Sym_m$.
	Then: 
	\begin{enumerate}
		\item The following conditions are equivalent:
			\begin{enumerate}
				\item\label{pt1-281021-2128} $\cM\succeq 0$.
				\item\label{pt2-281021-2128} $C\succeq 0$, $\cC B^T\subseteq\cC C$ and $\cM/C\succeq 0$.
				\item\label{pt3-281021-2128} $A\succeq 0$, $\cC B\subseteq\cC A $ and $\cM/A\succeq 0$.
			\end{enumerate}
            \smallskip
		\item\label{021121-1052} If $\cM\succeq 0$, then $\Rank \cM=\Rank A$ if and only if $\cM/A=0$.
            \smallskip
		\item\label{170622-1154} If $\cM\succeq 0$, then $\Rank \cM=\Rank C$ if and only if $\cM/C=0$.
	\end{enumerate}
\end{theorem}




\section{The $K$--flat extension theorem for a truncated univariate sequence}
\label{sec:K-flat-extension}

In this section we extend the abstract solution to the \CB{truncated matrix-valued moment problem} for a semialgebraic set $K$ from the scalar \cite[Theorem 5.2]{CF00} (see also \cite[Theorem 1.6]{Lau05}) to the matrix case.

\begin{theorem}
	\label{main-theorem-060525}
	Let $n,p\in\NN$, $K$ 
        be a closed nonempty semialgebraic set such that $K=K_{S}$, where
        $S=\{g_1,\ldots,g_k\}\subset \RR[\x]$,
	and 
	$$
        \Gamma\equiv \Gamma^{(2n)}
        =(\Gamma_0,\Gamma_1,\ldots,\Gamma_{2n})\in (\Sym_p)^{2n+1}
        $$
        be a given sequence. 
    Let $v_j:=\lceil\frac{\deg g_j}{2}\rceil$ and $v:=\max(\max_j v_j,1)$.
	Then the following statements are equivalent:
	\begin{enumerate}
	\item\label{main-theorem-060525-pt1}
            There exists a $(\Rank \cM_{n-v})$--atomic $K$--representing matrix measure $\mu$ for $\Gamma$.
        \smallskip
	\item\label{main-theorem-060525-pt2}
            The following statemets hold:
            \smallskip
            \begin{enumerate}
                \item\label{main-theorem-060525-pt2-ass1} 
                    $\cM_n\succeq 0$.
                \smallskip
                \item\label{main-theorem-060525-pt2-ass2} 
                    $\cH_{g_j}(n-v_j)\succeq 0$ for $j=1,\ldots,k$.
                \smallskip
                \item\label{main-theorem-060525-pt2-ass3} 
                    $\Rank \cM_{n-v}=\Rank \cM_{n}.$
            \end{enumerate}
	\end{enumerate}
    
    Moreover, $\mu$
    has $\Rank \cM_{n-v}-\Rank \cH_{g_j}(n-v)$
    atoms $x\in \RR$ that satisfy $g_j(x)=0$.
\end{theorem}

\begin{proof}
The nontrivial implication of the theorem is
    $
    \eqref{main-theorem-060525-pt2}
    \Rightarrow
    \eqref{main-theorem-060525-pt1}.
    $
Let $r:=\Rank \cM_{n-v}$. 
By \cite[Theorem 2.7.9]{BW11}, there is $C\in \RR^{p\times r}$ 
and a $r\times r$ diagonal matrix $D=\diag(d_1,\ldots,d_r)$ such that $\Gamma_i=CD^iC^T$
for $i=0,\ldots,2n$.  
It remains to show that $d_1,\ldots,d_r\in K$.
Let 
$$V_l=\begin{pmatrix}
    C^T &
    DC^{T} &
    \cdots &
    D^{l}C^{T}
\end{pmatrix} \quad \text{for} \quad l=0,\ldots,n+1.$$
Note that
\begin{equation}
\label{factorization}
    \cM_{l}
    =
    V_l^T V_l
\quad \text{for} \quad l=0,\ldots,n.
\end{equation}
By assumption \eqref{main-theorem-060525-pt2-ass3},
it follows that $\Rank V_{l}=r$ for $l=n-v,\ldots,n$.
Further, for each $g_j$, $j=1,\ldots,k$,
we have that
\begin{align}
\label{eq:loc-mat-0-0605}
\begin{split}
0\preceq \cH_{g_j}(n-v_j)
    &=V_{n-v_j}^Tg_j(D)V_{n-v_j}
    =V_{n-v_j}^T\diag(g_j(d_i))_{i=1}^r V_{n-v_j}.
\end{split}
\end{align}
Since $\Rank V_{n-v_j}=r$, it follows from 
\eqref{eq:loc-mat-0-0605} that 
for each $j$ and each $i$,
    $g_j(d_i)\geq 0$.
This proves that $d_j\in K$ for each $j$. Denoting the $i$--th column of $C$ by $\mbf{c}_i$,
we have that the $K$--\CB{representing matrix-valued measure} for $\Gamma$ is equal to
$\sum_{i=1}^{r} \mbf{c}_i\mbf{c}_i^T\delta_{d_i}$.
This concludes the proof of the implication
    $
    \eqref{main-theorem-060525-pt2}
    \Rightarrow
    \eqref{main-theorem-060525-pt1}.
    $

Let us prove the moreover part. 
Replacing $v_j$ with $v$ in \eqref{eq:loc-mat-0-0605} we see that
\begin{align*}
\Rank \cH_{g_j}(n-v)
&=
r-|\{d_i\colon g_j(d_i)=0\text{ and }i\in \{1,\ldots,r\}\}|\\
&=
\Rank \cM_{n-v}-|\{d_i\colon g_j(d_i)=0\text{ and }i\in \{1,\ldots,r\}\}|,
\end{align*}
which implies the moreover part.
\end{proof}

\begin{remark}
\begin{enumerate}
\item
Note that in Theorem \ref{main-theorem-060525},
$v>\max_j v_j$ if and only if $v_j=0$ for every $j$. But then each $g_j$ is a nonnegative constant (since $K\neq \emptyset$) and $K=\RR$.
Then Theorem \ref{main-theorem-060525} is \cite[Theorem 2.7.9]{BW11}.
\smallskip
\item 
A $d$--variate version of Theorem \ref{main-theorem-060525}
for $K=\RR^d$ and $S=\{1\}$ is \cite[Proposition 4.3]{MS16} (see also \cite[Theorem 6.2]{KT22} and
\cite[Theorem 4.3]{MS23++}).
\end{enumerate}
\end{remark}

\section{Coflatness implies flatness}
\label{sec:coflatness}

Let $n\in \NN$, $L:\RR[\x]_{\leq n}\to \Sym_p$ be a 
linear mapping and $f\in \RR[\x]_{\leq n}$.
We call:
\begin{enumerate} 
\item A moment matrix $\cM_{\lfloor \frac{n}{2}\rfloor}$ of $L$
 \textbf{coflat}, if 
$\Rank \cM_{\lfloor \frac{n}{2}\rfloor}=\Rank \cH_{\x^2}(\lfloor \frac{n}{2}\rfloor-1)$. 
\smallskip
\item An $f$--localizing moment matrix $\cH_f$ \textbf{coflat}, if
$\Rank \cH_{f}=\Rank \cH_{\x^2f}.$
\end{enumerate}

The main result of this section, Proposition \ref{lem:aux-3}, states that for a compact set $K\subseteq \RR$ with $0=\min K$, under the positivity assumptions, the coflatness of the moment matrix (resp.\ the $(\max K-\x)$--localizing matrix) implies its flatness. This will be essentially used in the solution to the \CB{truncated matrix-valued moment problem} for a union of an interval and a point, since the coflatness can be easily achieved by manipulating $L(1)$. 

\begin{proposition}
    \label{lem:aux-3}
	Let \CB{$K\subset \RR$ be a finite union of closed intervals and points}
    with $\min K=0$ and 
        $\max K=c$, $c>0$.
        Let 
	$n\in \NN$, $p\in \NN$
	and
		$\Gamma=(\Gamma_0,\ldots,\Gamma_{n})\in (\Sym_p)^{n+1}$
	be a sequence with the Riesz mapping $L_\Gamma$. Assume that 
        $\cM_m$ is positive semidefinite.
    If:
    \smallskip
    \begin{enumerate}
        \item\label{lem:aux-3-pt1} $n=2m$,
            $\cH_{\x(c-\x)}(m-1)$ is positive semidefinite
            and 
            $\cM_m$ is coflat, then
	       \begin{equation}
                \label{lem:aux-3-pt1-eq}
	            \Rank \cM_{m}=\Rank \cM_{m-1}.
               \end{equation}
        \item\label{lem:aux-3-pt2} $n=2m+1$,
            $\cH_{\x}(m)$ and 
            $\cH_{c-\x}(m)$ are positive semidefinite,
            and 
            $\cH_{c-\x}$ is coflat,
            then
        \begin{equation}
                \label{lem:aux-3-pt2-eq}
	               \Rank \cH_{c-\x}(m)=\Rank \cH_{c-\x}(m-1).
          \end{equation}
    \end{enumerate}
\end{proposition}

\bigskip

In the proof we will need a few lemmas.

\begin{lemma}
    \label{lem:aux-1}
    Let  $\Gamma\equiv(\Gamma_0,\Gamma_1,\ldots,\Gamma_n)\in     
            \big(\Sym_p\big)^{n+1}$
    be a sequence, $f\in \RR[\x]_{\leq n}$ a polynomial
    and $m\in \ZZ_+$, $m\leq \frac{1}{2}(n-\deg{f})$.
    Then 
    \begin{equation}
    \label{040525-1854}
        B_m(t) \cH_f(m)\big(B_m(t)\big)^T
        =
        \cH_{(\x-t)^2f(\x)}(m-1),    
    \end{equation}
where
    \begin{equation*} 
		B_m(t):=
		\begin{pmatrix}
		-tI_p & I_p & 0  & \cdots & 0\\
		0 & -tI_p & I_p& \ddots & \vdots \\
		\vdots  & \ddots & \ddots &  \ddots & 0 \\
		0 & \cdots & 0 &  -tI_p & I_p
		\end{pmatrix}\in M_{m,m+1}(M_p)
   \end{equation*}
is a $m\times (m+1)$ block matrix with blocks of size $p\times p$.
\end{lemma}

\begin{proof}
First we will show that 
\begin{equation} 
    \label{040525-1853}
        (B_m(t)\cH_f(m))_{ij}=\Gamma^{((\x-t)f(\x))}_{i+j-2}
        \quad 
        \text{for}
        \quad 
        1\leq i\leq m,\;
        1\leq j\leq m+1,
    \end{equation}
where $(B_m(t)\cH_f(m))_{ij}$ stands for the matrix block in the $i$--th row and $j$--th column of $B_m(t)\cH_f(m)$.
By definition of $\Gamma^{((\mathbf{x}-t)f)}_{i}$, to establish \eqref{040525-1853}
we need to prove that
\begin{equation}
\label{eq:localizing}
        (B_m(t)\cH_f(m))_{ij}=L\big((\x-t)f(\x)\x^{i+j-2}\big)
        \quad 
        \text{for each}
        \quad 
        1\leq i\leq m,\;
        1\leq j\leq m+1,
\end{equation}
where $L$ is the Riesz functional of $\Gamma$.
Let $(f\cdot\mbf{v})_i^{(j)}$ be as in \eqref{def:columns-loc-mm}.
We have that:
\begin{align*}
    (B_m(t)\cH_f(m))_{ij}
    &=
        \Big(B_m(t)
        \begin{pmatrix} 
            (f\cdot\mathbf{v})_{0}^{(m)} &
            \cdots
            (f\cdot\mathbf{v})_{j-1}^{(m)} &
            \cdots &
            (f\cdot\mathbf{v})_{m}^{(m)} &
        \end{pmatrix}\Big)_{ij}\\
    &=
        \sum_{\ell=1}^{m+1} \big(B_m(t)\big)_{i\ell}\Gamma^{(f)}_{\ell+j-2}\\[0.2em]
    &=-t\Gamma_{i+j-2}^{(f)}+\Gamma_{i+j-1}^{(f)}
    \\[0.2em]
    &=
        L\big(-tf(\x)\x^{i+j-2}+f(\x)\x^{i+j-1}\big)\\[0.2em]
    &=
        L\Big((\x-t)f(\x)
        \x^{i+j-2}\Big),
\end{align*}
which is \eqref{eq:localizing}.

Finally, \eqref{040525-1854} follows by the following computation:
\begin{align*}
    B_m(t)\cH_f(m)\big(B_m(t)\big)^T
    &\underbrace{=}_{\eqref{eq:localizing}}(\Gamma^{((\x-t)f(\x))}_{i+j-2})_{  
        \substack{
            1\leq i\leq m,\\
            1\leq j\leq m+1
            }} \big(B_m(t)\big)^T\\[0.2em]
    &=(-t\Gamma^{((\x-t)f(\x))}_{i+j-2}+\Gamma^{((\x-t)f(\x))}_{i+j-1})_{  
        \substack{
            1\leq i\leq m,\\
            1\leq j\leq m
            }} \\[0.2em]
    &=(\Gamma^{(-t(\x-t)f(\x))}_{i+j-2}+\Gamma^{(\x(\x-t)f(\x))}_{i+j-2})_{  
        \substack{
            1\leq i\leq m,\\
            1\leq j\leq m
            }} \\[0.2em]
    &=(\Gamma^{((\x-t)^2f(\x))}_{i+j-2})_{  
        \substack{
            1\leq i\leq m,\\
            1\leq j\leq m
            }} \\[0.2em]
    &=\cH_{(\x-t)^2f(\x)}(m-1),
\end{align*}
which  concludes the proof of the lemma.
\end{proof}

The following lemma states that the rank of a matrix is a monotone function on the set of \CB{positive semidefinite} matrices with respect
to the usual Loewner order, i.e., $A\succeq B$ if and only if $A-B\succeq 0$.

\begin{lemma}\label{110722-1428}
	Let \CB{$p\in \NN$} and $A,B\in \CB{\Sym_p}$ such that $A\succeq B\succeq 0$. Then $\cC B\subseteq \cC A$ and $\Rank A\geq \Rank B.$
\end{lemma}

\begin{proof} 
	Since for every $X\in \CB{\Sym_p}$ it holds that $\cC X$ is an orthogonal complement of $\ker X$ 
	with respect to the usual Euclidean inner product, it suffices to prove that $\ker A\subseteq \ker B$.
	Let us take $v\in \ker A$. 
	From $0=v^TAv\geq v^TBv\geq 0$, it follows that $v^TBv=0$. 
	By
		$0=v^TBv=v^TB^{\frac{1}{2}}B^{\frac{1}{2}} v=\|B^{\frac{1}{2}} v\|^2,$
	it follows that $v\in \ker B^{\frac{1}{2}}$ and thus $v\in \ker B.$
\end{proof}
\bigskip

Now we are ready to prove 
Proposition \ref{lem:aux-3}.

\begin{proof}[Proof of Proposition \ref{lem:aux-3}]
First we prove \eqref{lem:aux-3-pt1}.
The rank equality \eqref{lem:aux-3-pt1-eq} will follow
once we establish the following two equalities:
\begin{align}\label{150722-1508}
    \cC\cM_{m-1}
        &= \cC\cH_{\x}(m-1)
        =\cC\cH_{\x^2}(m-1).
\end{align}

	From 
		$$
		0
		\preceq 
		\cM_{m}
		=
		\begin{pmatrix}
			 \cM_{m-1} & \mbf{v}^{(m-1)}_m\\[0.5em]
			(\mbf{v}^{(m-1)}_m)^T & \Gamma_{2m}
		\end{pmatrix}
            =
            \begin{pmatrix}
			 \mbf{v}_0^{(m-1)} & \cH_{\x}(m-1)\\[0.5em]
			 \Gamma_{m} & (\mbf{v}^{(m-1)}_{m+1})^T
		\end{pmatrix},
		$$ 
	it follows by Theorem \ref{block-psd}.\eqref{pt3-281021-2128} 
	used for the pair 
		$(\cM,A)=(\cM_m,\cM_{m-1})$,
	that 
	\begin{equation}\label{150722-1507}
		\cC\cM_{m-1}\supseteq \cC\cH_\x(m-1).
	\end{equation}
	Similary, from 
		$$
		0
		\preceq 
		\cM_{m}
		=
		\begin{pmatrix}
			\Gamma_0 &  (\mbf{v}_1^{(m-1)})^T\\[0.5em]
			\mbf{v}_1^{(m-1)}	& \cH_{\x^2}(m-1)
		\end{pmatrix},
		$$ 
	it follows by Theorem \ref{block-psd}.\eqref{pt2-281021-2128}
	used for the pair 
		$(\cM,C)=(\cM_m,\cH_{\x^2}(m-1))$,
	that 
	\begin{equation}\label{160722-2341}
	       \cC\mathbf v_1^{(m-1)}\subseteq \cC\cH_{\x^2}(m-1).
	\end{equation}
    The assumption $\Rank \cM_m=\Rank \cH_{\x^2}(m-1)$, \CB{in particular implies that
    \begin{equation} 
    \label{additional-inclusion}
    \cC \mathbf{v}_0^{(m-1)}\subseteq 
    \cC 
    \begin{pmatrix}
        \mathbf{v}_1^{(m-1)} & \ldots & \mathbf{v}_m^{(m-1)}
    \end{pmatrix}.
    \end{equation}
    By \eqref{160722-2341}, \eqref{additional-inclusion} and
    $$\cM_{m-1}=
    \begin{pmatrix}
        \mathbf{v}_0^{(m-1)} & 
        \ldots &
        \mathbf{v}_{m-1}^{(m-1)}
    \end{pmatrix},\quad
    \cH_{\x^2}(m-1)=
    \begin{pmatrix}
        \mathbf{v}_2^{(m-1)} & 
        \ldots &
        \mathbf{v}_{m+1}^{(m-1)}
    \end{pmatrix},
    $$} it follows that
    \begin{equation}
        \label{030525-1800}
            \cC\cM_{m-1}\subseteq \cC\cH_{\x^2}(m-1).
    \end{equation}
    Hence, by \eqref{150722-1507}--\eqref{030525-1800} we have that
    \begin{equation}
        \label{030525-1802}
                \cC\cH_{\x}(m-1)
            \subseteq 
                \cC\cM_{m-1}
            \subseteq 
                \cC\cH_{\x^2}(m-1).
    \end{equation} 
    Note that		
	\begin{equation}\label{270722-0930}
		\frac{1}{c}
            \big(
                \cH_{\x(c-\x)}(m-1)
                +
                \cH_{\x^2}(m-1)
            \big)=\cH_{\x}(m-1).
	\end{equation}	
	Since $\cH_{\x(c-\x)}(m-1)$ and $\cH_{\x^2}(m-1)$ are both \CB{positive semidefinite}, it follows from \eqref{270722-0930} that 
		$$0\preceq \cH_{\x^2}(m-1)\preceq \cH_{\x(c-\x)}(m-1)+\cH_{\x^2}(m-1).$$
	Therefore
	\begin{equation}\label{150722-1303}
		\cC\cH_{\x^2}(m-1)
			\subseteq 
				\cC\big(\cH_{\x(c-\x)}(m-1)+\cH_{\x^2}(m-1)\big)
			= 
				\cC\cH_{\x}(m-1),
	\end{equation}
	where the inclusion follows by Lemma \ref{110722-1428} used for 
		$A=\cH_{\x(c-\x)}(m-1)+\cH_{\x^2}(m-1)$ 
	and
		$B=\cH_{\x^2}(m-1)$,
	while the equality follows by \eqref{270722-0930}.
    Now \eqref{150722-1303}
    implies that all inclusions in 
    \eqref{030525-1802} are equalities,
    which proves \eqref{150722-1508}
    and concludes the proof of the part \eqref{lem:aux-3-pt1}.\\

    It remains to prove \eqref{lem:aux-3-pt2}.
    We have that 
	\begin{equation}\label{270722-2225}
		\Rank \cH_{c-\x}(m)=\Rank \cH_{\x^2(c-\x)}(m-1)
		\leq \Rank \cH_{\x(c-\x)}(m-1),
	\end{equation}
	where the equality follows by the assumption, while the inequality follows by the equality
    \begin{align*}
        \cH_{\mathbf x(c-\mathbf x)}(m-1)
		&=\frac{1}{c}
            \big( 
                \cH_{(c-\x)^2\x}(m-1)+\cH_{\x^2(c-\x)}(m-1)
            \big)\\
        &\underbrace{=}_{\eqref{040525-1854}}
                \frac{1}{c}
                \Big(
                    \underbrace{B_m(c)\cH_{\mathbf x}(m)\big(B_m(c)\big)^T}_{\succeq 0}
                    +
                    \underbrace{B_m(0)\cH_{c-\mathbf{x}}(m)\big(B_m(0)\big)^T}_{\succeq 0}
                \Big)
    \end{align*}
	and Lemma \ref{110722-1428}.
   Similarly, the inequality
    \begin{align*} 
        \cH_{c-\mathbf x}(m-1)
		&=\frac{1}{c}
            \big(
                \cH_{(c-\x)^2}(m-1)+\cH_{\x(c-\x)}(m-1)
            \big)\\
        &\underbrace{=}_{\eqref{040525-1854}}
                \frac{1}{c}
                    \Big(
                    \underbrace{B_m(c)\cM(m)\big(B_m(c)\big)^T}_{\succeq 0}
                    +
                    \underbrace{\cH_{\mathbf x(c-\mathbf{x})}(m-1)}_{\succeq 0}
                    \Big)
    \end{align*}
	and Lemma \ref{110722-1428}
    imply that 
    \begin{equation}
        \label{030525-2058}
            \Rank \cH_{\x(c-\x)}(m-1)
            \leq
            \Rank \cH_{c-\x}(m-1).
    \end{equation}
    The inequalities \eqref{270722-2225} and \eqref{030525-2058}
    imply that
    $$
    \Rank \cH_{c-\x}(m)\leq \Rank \cH_{c-\x}(m-1),
    $$
    which is only possible in the case of the equality, whence 
    \eqref{lem:aux-3-pt2-eq} holds and concludes the proof
    of the part
    \eqref{lem:aux-3-pt2}.
\end{proof}


\section{\CB{The truncated matrix-valued moment problem} on a union of an interval and a point}
\label{sec:union-interval-point}

In this section we solve the \CB{truncated matrix-valued moment problem on a union of an interval and a point}. Then we use this solution to prove Theorem \ref{thm:bounded-interval}.

\begin{theorem}
	\label{main-theorem-1}
	Let $n,p\in\NN$, $a,b,c\in \RR$, $a<b<c$,
        $$K=K_S=\{a\}\cup [b,c],$$
        where
        $S:=\{f_1,f_2,f_3\}$
        with
        $f_1(\x)=\x-a$, $f_2(\x)=(\x-a)(\x-b)$, $f_3(\x)=c-\x$,
	and 
	$$
        \Gamma\equiv \Gamma^{(n)}
        =(\Gamma_0,\Gamma_1,\ldots,\Gamma_n)\in (\Sym_p)^{n+1}
        $$
        be a given sequence. 
	Then the following statements are equivalent:
	\begin{enumerate}
	\item\label{main-theorem-1-pt1}
            There exists a $K$--representing matrix measure for $\Gamma$.
        \smallskip
	\item\label{main-theorem-1-pt2}
            There exists a finitely--atomic $K$--representing matrix measure for $\Gamma$.
        \smallskip
	\item\label{main-theorem-1-pt3}
            One of the following statements holds:
            \smallskip
	\begin{enumerate}
	\item\label{main-theorem-1-pt3a} $n=2m$ for some $m\in\NN$ and
		\begin{align}
            \label{loc-mat-1}
		\cM_m
            \succeq 0,
            \quad
            \cH_{f_2}(m-1)\succeq 0,
            \quad\text{and}\quad
            \cH_{f_1f_3}(m-1)&\succeq 0.
		\end{align}
	\item\label{main-theorem-1-pt3b} $n=2m+1$ for some $m\in \NN$ and
		\begin{align}
            \label{loc-mat-2}
		\cH_{f_1}(m)
            \succeq 0,
            \quad
            \cH_{f_3}(m)\succeq 0,
            \quad
            \cH_{f_1f_2}(m-1)&\succeq 0
            \quad\text{and}\quad
            \cH_{f_2f_3}(m-1)
        \succeq 0.
        \end{align}
	\end{enumerate}
	\end{enumerate}
    
    Moreover, if $n=2m$, then there is
    a $(\Rank \cM_{n})$--atomic $K$--representing measure for $\Gamma$,
    while if $n=2m+1$, there exists at most $(\Rank \cM_{n}+p)$--atomic  one.
\end{theorem}
\bigskip

\begin{proof}
The nontrivial implication is $\eqref{main-theorem-1-pt3}\Rightarrow\eqref{main-theorem-1-pt2}.$
By applying an an affine linear transformation of variables we may assume that $a=0$, $b=1$ and $c>1$.
Hence, $f_1(\x)=\x$, $f_2(\x)=\x(\x-1)$, $f_3(\x)=c-\x$.

First assume that $n=2m$, $m\in \NN$. Note that $\Gamma_0$ only appears in $\cM_m$,
but not in any of $\cH_{\x(\x-1)}(m-1)$, $\cH_{\x(c-\x)}(m-1)$.
Let us replace $\Gamma_0$ by the smallest $\widetilde\Gamma_0$ such that $\widetilde \cM_m\succeq 0$,
where $\widetilde \cM_\ell$ is the moment matrix corresponding to 
    $\widetilde \Gamma\equiv (\widetilde \Gamma_0,\Gamma_1,\ldots,\Gamma_{2\ell})$, $1\leq\ell\leq m$.
Namely, by Theorem \ref{block-psd}, used for the pair $(\cM,C)=(\widetilde\cM_m,\cH_{\x^2}(m-1))$,
we have that 
$$
\widetilde \Gamma_0=  
    \begin{pmatrix}
        \Gamma_1 & \cdots & \Gamma_m
    \end{pmatrix}
    \big(\cH_{\x^2}(m-1)\big)^\dagger 
    \begin{pmatrix}
        \Gamma_1 \\ \cdots \\ \Gamma_m
    \end{pmatrix}
$$
and
$$
    \Rank \widetilde{\cM}_{m}=\Rank \cH_{\x^2}(m-1).
$$
By Proposition \ref{lem:aux-3}, we have that $\Rank \widetilde \cM_m=\Rank \widetilde \cM_{m-1}$.
By Theorem \ref{main-theorem-060525}, it follows that 
$\widetilde \Gamma$ has a $K$--\CB{representing matrix-valued measure}
of the form 
$\sum_{i=1}^{r} \mbf{c}_i\mbf{c}_i^T\delta_{d_i},$
where $r=\Rank \widetilde \cM_m$ and $c_i\in \RR^{p}$.
Then
$\sum_{i=1}^{r} \mbf{c}_i\mbf{c}_i^T\delta_{d_i}+(\Gamma_0-\widetilde \Gamma_0)\delta_{0}$ is a $(\Rank \cM_m)$--atomic $K$--\CB{representing matrix-valued measure} for $\Gamma$.
This proves $\eqref{main-theorem-1-pt3a}\Rightarrow\eqref{main-theorem-1-pt2}$.

Now assume that $n=2m+1$, $m\in \NN$. Note that $\Gamma_0$ only appears in $\cH_{c-\x}(m)$,
but not in any of $\cH_{\x}(m)$, $\cH_{\x^2(\x-1)}(m-1)$, $\cH_{\x(\x-1)(c-\x)}(m-1)$.
Let us replace $\Gamma_0$ by the smallest $\widetilde\Gamma_0$ such that $\widetilde \cH_{c-\x}(m)\succeq 0$,
where $\widetilde \cH_{c-\x}(m)$ is the moment matrix corresponding to 
    $\widetilde \Gamma\equiv (\widetilde \Gamma_0,\Gamma_1,\ldots,\Gamma_{2m+1})$.
Below $\widetilde \cM_\ell$ and
$\widetilde\cH_{f}(\ell)$ will refer to the moment matrix and the $f$--localizing moment matrix of $\widetilde\Gamma$, respectively.
By Theorem \ref{block-psd}, used for the pair 
$$(\cM,C)=(\widetilde\cH_{c-\x}(m),\cH_{\x^2(c-\x)}(m-1)),$$
we have that 
\begin{align*}
\widetilde \Gamma_0
    &= 
    \frac{1}{c}
    \Gamma_1+ 
    \frac{1}{c}
    \begin{pmatrix}
        c\Gamma_1-\Gamma_2 & \cdots & c\Gamma_m-\Gamma_{m+1}
    \end{pmatrix}
    \big(\cH_{\x^2(c-\x)}(m-1)\big)^\dagger
    \begin{pmatrix}
        c\Gamma_1-\Gamma_2 \\ \cdots \\ c\Gamma_m-\Gamma_{m+1}
    \end{pmatrix}
\end{align*}
and
\begin{align*}
\Rank \widetilde\cH_{c-\x}(m)
    &=\Rank \cH_{\x^2(c-\x)}(m-1).
\end{align*}
By Proposition \ref{lem:aux-3}, we have that $\Rank \widetilde\cH_{c-\x}(m)=\Rank \widetilde\cH_{c-\x}(m-1)$.
Let $Q_0,\ldots,Q_{m-1}\in M_p(\RR)$ be such that 
\begin{align} 
\notag
    c\Gamma_{m}-\Gamma_{m+1}
    &=(c\widetilde\Gamma_0-\Gamma_1) Q_0+\sum_{i=1}^{m-1}(c\Gamma_i-\Gamma_{i+1}) Q_i,\\
\label{RG-relation-2-v2}
    c\Gamma_{m+j}-\Gamma_{m+j+1}
    &=\sum_{i=0}^{m-1}(c\Gamma_{i+j}-\Gamma_{i+1+j}) Q_i,\quad j=1,\ldots,m. 
\end{align}
Equivalently, defining 
\begin{align*}
    \widetilde Q_0&:=-cQ_0,\\
    \widetilde Q_i&:=Q_{i-1}-cQ_i\quad 
        \text{for}\quad i=1,\ldots,m-1,\\
    \widetilde Q_m&:=cI+Q_{m-1},
\end{align*}
where $I$ is the identity matrix,
we have that
\begin{align} 
\notag
    \Gamma_{m+1}
    &=
    \widetilde\Gamma_0\widetilde Q_0+\sum_{i=1}^{m} \Gamma_i \widetilde Q_i,\\
\label{RG-relation-2}
    \Gamma_{m+j+1}
    &=
    \Gamma_j\widetilde Q_0+\sum_{i=1}^m \Gamma_{i+j} \widetilde Q_i, \quad j=1,\ldots,m. 
\end{align}
Let $\widetilde \cM_\ell$ stand for the $\ell$--th truncated moment matrix of 
$\widetilde \Gamma$.
Observe that
    $$
        \widetilde\cM_m=
            \frac{1}{c}
            \left(
                \widetilde\cH_{f_1}(m)+\widetilde\cH_{f_3}(m)
            \right),$$
whence $\widetilde\cM_m\succeq 0$.
Let us define $\Gamma_{2m+2}, \Gamma_{2m+3}, \Gamma_{2m+4}$ by \eqref{RG-relation-2} used for $j=m+1,m+2,m+3$.
Note that 
\begin{align} 
\label{RG-v4}
\Gamma_{m+3+j}-\Gamma_{m+2+j}
    &=\sum_{i=0}^{m} (\Gamma_{i+2+j}-\Gamma_{i+1+j})\widetilde Q_i
        \quad \text{for}\quad j=0,\ldots,m+1,\\
\label{RG-v3}
    c\Gamma_{m+2+j}-\Gamma_{m+3+j}
    &=\sum_{i=0}^{m} (c\Gamma_{i+1+j}-\Gamma_{i+2+j})\widetilde Q_i
        \quad \text{for}\quad j=0,\ldots,m+1.
\end{align}
By definition of $\Gamma_{2m+2}$ we have that $\Rank\widetilde \cM_{m+1}=\Rank \widetilde\cM_m$.
If $\Gamma_{2m+2}\neq \widetilde \cM_{m+1}/\widetilde\cM_m$, then
 the equality of ranks cannot hold, whence 
$\Gamma_{2m+2}=\widetilde \cM_{m+1}/\widetilde\cM_m$,
$\Gamma_{2m+2}$ is symmetric
and $\widetilde\cM_{m+1}\succeq 0$. Similarly,
by \eqref{RG-relation-2-v2} and
\eqref{RG-v3} we have that 
 $\Rank\widetilde \cH_{c-\x}(m+1)=\Rank \widetilde\cH_{c-\x}(m)$, whence $\Gamma_{2m+3}$ is symmetric,
 $\widetilde\cH_{c-\x}(m+1)\succeq 0$
 and
 $\widetilde\cH_{\x^2(c-\x)}(m+1)\succeq 0$.
Further, by definition of $\Gamma_{2m+4}$ we have that $\Rank\widetilde \cM_{m+2}=\Rank \widetilde\cM_{m+1}$, whence 
$\Gamma_{2m+4}$ is symmetric
and $\widetilde \cM_{m+2}\succeq 0$.
By \eqref{RG-relation-2}, it follows that
$\Rank\widetilde\cH_{\x}(m+1)=\Rank \widetilde\cH_{\x}(m)$,
whence $\widetilde\cH_{\x}(m+1)\succeq 0$.
Finally, by \eqref{RG-v4}, we have that
$\Rank\widetilde\cH_{\x^2(\x-1)}(m)=\Rank \widetilde\cH_{\x^2(\x-1)}(m-1)$,
whence $\widetilde\cH_{\x^2(\x-1)}(m)\succeq 0$.
By Theorem \ref{main-theorem-060525}, it follows that 
$\widetilde \Gamma$ has a $K$--\CB{representing matrix-valued measure}
of the form 
$\sum_{i=1}^{r} \mbf{c}_i\mbf{c}_i^T\delta_{d_i},$
where $r=\Rank \widetilde \cM_m$ and $c_i\in \RR^{p}$.
Then
$\sum_{i=1}^{r} \mbf{c}_i\mbf{c}_i^T\delta_{d_i}+(\Gamma_0-\widetilde \Gamma_0)\delta_{0}$ is at most a 
$(\Rank \widetilde\cM_m+\Rank \Gamma_0-\widetilde \Gamma_0)$--atomic $K$--\CB{representing matrix-valued measure} for $\Gamma$.
This proves $\eqref{main-theorem-1-pt3b}\Rightarrow\eqref{main-theorem-1-pt2}$.
\end{proof}
\bigskip

Now we are ready to prove Theorem \ref{thm:bounded-interval}. 

\begin{proof}[Proof of Theorem \ref{thm:bounded-interval}] 
    Assume the notation from Theorem \ref{main-theorem-1}
    and let 
        $$
        \widetilde S=
        \left\{
        \begin{array}{rl}
            \{1,f_2,f_1f_3\}
            ,&   \text{if }n\text{ is even},\\[0.2em]
            \{f_1,f_3,f_1f_2,f_2f_3\},&   \text{if }n\text{ is odd}.
        \end{array}
        \right.
        $$
    We have to prove that
    \begin{equation}
    \label{eq-to-prove}
        \Pos_n^{(p)}(\{a\}\cup [b,c])=\QM_{\widetilde S,n}^{(p)}.
    \end{equation}
This follows by using
Theorem \ref{main-theorem-1} and 
Proposition \ref{prop:140525} for $K=\{a\}\cup [b,c]$.

Namely, assume that $n=2m$. 
Note that 
\begin{align}
\label{equivalence-140525}
\begin{split}
    &\cM_m\succeq 0\\
    \Leftrightarrow& \;\;
    \langle \cM_m,B \rangle\geq 0
    \;\; \text{for every }
        B\in \Sym_{(m+1)p}^{\succeq 0}\\[0.2em]
    \Leftrightarrow& \;\;
    \langle \cM_m,
        \widetilde B \widetilde B^T \rangle\geq 0
    \;\; \text{for every }
        \widetilde B=
        (\widetilde B_i)_{i=0}^{m}\in
        (M_p(\RR))^{m+1}\\[0.2em]
    \Leftrightarrow& \;\;
    \sum_{i,j=0}^m \tr(\widetilde B_i^T \Gamma_{i+j} \widetilde B_j)\geq 0
    \;\; \text{for every }
        \widetilde B=
        (\widetilde B_i)_{i=0}^{m}\in
        (M_p(\RR))^{m+1}\\[0.2em]
    \Leftrightarrow& \;\;
    \sum_{i,j=0}^m 
        \tr(\Gamma_{i+j} 
        \widetilde B_j
        \widetilde B_i^T )\geq 0
    \;\; \text{for every }
        \widetilde B=
        (\widetilde B_i)_{i=0}^{m}\in
        (M_p(\RR))^{m+1}\\[0.2em]
    \Leftrightarrow& \;\;
    \sum_{k=0}^n \tr(\Gamma_k A_k)\geq 0
    \;\; \text{for every }
        \sum_{i=0}^{n} A_{i}\x^i
        =
        \big(\sum_{j=0}^m \widetilde B_j \x^j\big)
        \big(\sum_{j=0}^m \widetilde B_j \x^j\big)^T
        \in 
            M_p(\RR[\x]_{\leq n})\\[0.2em]
    \Leftrightarrow& \;\;
    \sum_{k=0}^n \tr(\Gamma_k A_k)\geq 0
    \;\; \text{for every }
        \sum_{i=0}^{n} A_{i}\x^i
        \in 
            \sum M_p(\RR[\x])^2.
\end{split}
\end{align}
where the second equivalence follows by noting that every $B\in \Sym_{(m+1)p}^{\succeq 0}$ is a sum of the form $\widetilde B \widetilde B^T$
with $\widetilde B\in (M_p(\RR))^{m+1}$,
while the third equivalence follows by definition of the inner product and $\cM_m$.
Similarly, for 
    $$f:=c_2\x^2+c_1\x+c_0\in \{f_2,f_1f_3\},$$
we have that
\begin{align}
\label{equivalence-2-140525}
\begin{split}
    &\;\; \cH_{f}(m-1)\succeq 0\\[0.2em]
    \Leftrightarrow&\;\; 
    \langle \cH_{f}(m-1,C \rangle\geq 0
    \;\; \text{for every }
        C\in \Sym_{mp}^{\succeq 0}\\[0.2em]
    \Leftrightarrow&\;\; 
    \langle \cH_{f}(m-1),
        \widetilde C^T \widetilde C \rangle\geq 0
    \;\; \text{for every }
        \widetilde C=
        (\widetilde C_i)_{i=0}^{m-1}\in
        (M_p(\RR))^{m}\\[0.2em]
    \Leftrightarrow&\;\;
    \sum_{k=0}^{n-2} \tr(\Gamma^{(f)}_k A_k)\geq 0
    \;\; \text{for every }
        \sum_{i=0}^{n-2} A_{i}\x^i
        \in 
            \sum M_p(\RR[\x])^2
            \\[0.2em]
    \Leftrightarrow&\;\;
    \sum_{k=0}^{n-2} 
        \tr\big(
            (\Gamma_{k+2} c_2+
            \Gamma_{k+1} c_1+ 
            \Gamma_{k} c_0)
            A_k\big)\geq 0
    \;\; \text{for every }
        \sum_{i=0}^{n-2} A_{i}\x^i
        \in 
            \sum M_p(\RR[\x])^2
            \\[0.2em]
    \Leftrightarrow&\;\;
    \sum_{k=0}^{n} 
        \tr(\Gamma_k \widetilde A_k)\geq 0
    \;\;\text{for every }
        \sum_{i=0}^{n} \widetilde A_{i}\x^i=
        f\big(\sum_{i=0}^{n-2} A_{i}\x^i\big)
        \;\;\text{with}\\[0.2em]
    &\hspace{6cm}\sum_{i=0}^{n-2} A_{i}\x^i
        \in 
            \sum M_p(\RR[\x])^2\\[0.2em]
    \Leftrightarrow&\;\;
    \sum_{k=0}^{n} 
        \tr(\Gamma_k \widetilde A_k)\geq 0
    \;\;\text{for every }
        \sum_{i=0}^{n} \widetilde A_{i}\x^i\in \QM_{\{f\},n}^{(p)},\\[0.2em]          
\end{split}
\end{align}
By Theorem \ref{main-theorem-1}, Proposition \ref{prop:140525} and \eqref{equivalence-140525}, \eqref{equivalence-2-140525},
we have that
\begin{align}
\label{equivalence-v3-1405}
\begin{split}
&\quad
    \sum_{k=0}^n \tr(\Gamma_k A_k)\geq 0
    \quad \text{for every }\sum_{i=0}^{n} A_{i}x^i\in \Pos^{(p)}_n(\{a\}\cup [b,c])\\
\Leftrightarrow&
\quad    
    \sum_{k=0}^n \tr(\Gamma_k A_k)\geq 0
    \quad \text{for every }\sum_{i=0}^{n} A_{i}x^i\in
        \QM_{\widetilde S,n}^{(p)}.
\end{split}
\end{align}
Since $\QM_{\widetilde S,n}^{(p)}$ is closed, \eqref{equivalence-v3-1405} implies \eqref{eq-to-prove}.
Indeed, if $\QM_{\widetilde S,n}^{(p)}\not\subseteq\Pos^{(p)}_n(\{a\}\cup [b,c])$, then there is 
$\sum_{k=0}^n \widetilde A_k\x^k\in \Pos^{(p)}_n\big(\{a\}\cup [b,c]\big)\setminus \QM_{\widetilde S,n}^{(p)}$.
By the Hahn-Banach theorem there is 
$\widetilde\Gamma:=(\widetilde\Gamma_0,\ldots,\widetilde \Gamma_n)$
such that 
$\sum_{k=0}^n \tr(\widetilde\Gamma_k \widetilde A_k)<0$ and 
$\sum_{k=0}^n \tr(\widetilde \Gamma_k A_k)\geq 0$ for every $\sum_{i=0}^n A_i\x^i\in \QM_{\widetilde S,n}^{(p)}$.
But this is a contradiction with \eqref{equivalence-v3-1405}.

The proof for $n$ of odd parity is analogous.
\end{proof}

\bigskip

\begin{remark}
\begin{enumerate}
\item 
By analogous reasoning as in this section one can obtain an alternative proof of the \CB{truncated matrix-valued moment problem on $[a,b]$} and the corresponding matrix Positivstellensatz. This proof seems to be more elementary than all of the existing proofs of this case (e.g.\ \cite{And70,DS02,CDFK06,DFKM10}).
\item 
The main reason why the approach in this section works is the fact that by subtracting the largest possible matrix mass in the isolated point, we obtain a coflat moment matrix in the case of even degree data and a coflat $(c-\x)$--localizing moment matrix in the case of odd degree data.
Then Proposition \ref{lem:aux-3} implies that the manipulated sequence is $K$--flat and the conclusion follows by Theorem \ref{main-theorem-060525}. In particular, this implies that if a $K$--\CB{representing matrix-valued measure} exists, there always exists a $K$--\CB{representing matrix-valued measure} with the largest matrix mass at the isolated point. If $K$ has more components, more localizing matrices bound from above the mass of the atom in some chosen point, and in general we cannot achieve coflatness. 
It would be very interesting to extend the approach to such $K$, by possibly achieving coflatness in a more involved way. 
\item 
It would be interesting to extend the results of this paper from matrix polynomials to operator polynomials where the coefficients are bounded operators on a Hilbert space.
The operator Fej\'er-Riesz theorem for $\RR$ is true in this setting \cite{Rosenblum} and also its version for a bounded or unbounded interval \cite[Proposition 3]{CZ13}, while it fails for $K$, which is a single point \cite[\S5.1]{CZ13}.
\end{enumerate}
\end{remark}


\CB{\section{Proof of Theorems \ref{thm:unbounded-interval} and \ref{thm:two-intervals}}
\label{sec:psatze-one-unbounded-interval}
In this section we prove Theorems \ref{thm:unbounded-interval} and \ref{thm:two-intervals},
using Theorem \ref{thm:bounded-interval} and elementary algebraic manipulations.

\subsection{Proof of Theorem \ref{thm:unbounded-interval}}
Let us assume, that $F$ is a matrix polynomial, positive semidefinite on
$K=\{a\}\cup [b,\infty)$, $a<b$. Denote $n:=\deg F$.
Define $\phi(x)=-\frac{1}{x-a+1}$ and note that $\phi(K)=\{-1\}\cup [-\frac{1}{b-a+1},0)$.
Then the polynomial 
    $$H(x):=(-x)^{n}F\big(-\frac{1}{x}+a-1\big)$$ 
is positive semidefinite on
$\overline{\phi(K)}=\{-1\}\cup [-\frac{1}{b-a+1},0]$.
Note that
\begin{equation} 
\label{express-F-with-G}
    F(x)=(x-a+1)^{n}H\big(-\frac{1}{x-a+1}\big).
\end{equation}
Expressing $F$ in the basis 
$(x-a+1)^{i}$, $i=0,\ldots,n$,
as
$F(x)=\sum_{i=0}^{n} F_i (x-a+1)^i$ with $F_i\in \Sym_p$, we have 
$$H(x)=\sum_{i=0}^{n} F_i(-x)^{n-i}=F_0(-x)^{n}+F_1(-x)^{n-1}+\ldots+F_{n}.$$
We separate two cases according to the parity of $\deg H$.\\

\noindent \textbf{Case 1:} $\deg H$ is even. 
By Theorem \ref{thm:bounded-interval}, there are matrix polynomials $H_0(x)$, $H_1(x)$, $H_2(x)$, each a sum of at most two symmetric squares, such that
\begin{equation} 
\label{certificate-G}
H(x)=H_0(\x)+(\x+1)\Big(\x+\frac{1}{b-a+1}\Big)H_1(\x)-(\x+1)\x H_2(\x),
\end{equation}
where the degree of each summand is at most $\deg H$.
Note that $n=2\lfloor n/2\rfloor+(n\;\text{mod}\; 2)$. 
Replacing $H$ in the equality \eqref{express-F-with-G}
with the right-hand side of \eqref{certificate-G}, we get
\begin{equation}
\label{F-in-terms-of-G}
    F(x)=(x-a+1)^{n\;\text{mod}\; 2}\big(K_0(x)+(x-a)(x-b)K_1(x)+(x-a)K_2(x)\big),
\end{equation}
where 
\begin{align*} 
    K_0(x)&=(x-a+1)^{2\lfloor n/2\rfloor}\;H_0\big(-\frac{1}{x-a+1}\big),\\
    K_1(x)&=\frac{1}{b-a+1}(x-a+1)^{2\lfloor n/2\rfloor-2}\;H_1\big(-\frac{1}{x-a+1}\big),\\
    K_2(x)&=(x-a+1)^{2\lfloor n/2\rfloor-2}\;H_2\big(-\frac{1}{x-a+1}\big)
\end{align*}
are matrix polynomials, positive semidefinite on $\RR$.
Indeed, \eqref{F-in-terms-of-G} follows from the fact that for
$$f_1(x):=(x+1)\big(\x+\frac{1}{b-a+1}\big)
\quad
\text{and}\quad 
f_2(x):=-(x+1)x,$$ 
we have
\begin{align*} 
(x-a+1)^2f_1\big(-\frac{1}{x-a+1}\big)&=\frac{(x-a)(x-b)}{b-a+1},\\
(x-a+1)^2f_2\big(-\frac{1}{x-a+1}\big)&=x-a.
\end{align*}
Using \eqref{F-in-terms-of-G}, it follows that 
\begin{equation}
\label{F-with-G-v3}
F(\x)=
        G_0(\x)+(\x-a)G_1(\x)+(\x-a)(\x-b)G_2(\x)+
        (\x-a)^2(\x-b)G_3(\x),
\end{equation}
where
\begin{align*}
G_0(x)
&:=
\left\{
    \begin{array}{rl}
        K_0(x),&   \text{if }n\;\text{mod}\;2=0,\\
        K_0(x)+(x-a)^2K_2(x),&   \text{if }n\;\text{mod}\;2=1,
    \end{array}
\right.\\
G_1(x)
&:=
\left\{
    \begin{array}{rl}
        K_2(x),&   \text{if }n\;\text{mod}\;2=0,\\
        K_0(x)+K_2(x),&   \text{if }n\;\text{mod}\;2=1,
    \end{array}
\right.\\
G_2(x)
&:=K_1(x),\\
G_3(x)
&:=
\left\{
    \begin{array}{rl}
        0,&   \text{if }n\;\text{mod}\;2=0,\\
        K_1(x),&   \text{if }n\;\text{mod}\;2=1.
    \end{array}
\right.
\end{align*}
Note that the degree of each summand on the right-hand side of \eqref{F-with-G-v3} is bounded 
by $\deg F$ and each $G_i(x)$ is positive semidefinite on $\RR$, whence by Theorem \ref{thm:Fejer-Riesz} it is a sum of at most two symmetric squares. This proves Theorem \ref{thm:unbounded-interval} for the case $\deg H$ is even.\\

\noindent \textbf{Case 2:} $\deg H$ is odd. 
By Theorem \ref{thm:bounded-interval}, there are matrix polynomials $H_0(x)$, $H_1(x)$, $H_2(x)$, $H_3(x)$, each a sum of at most two symmetric squares, such that
\begin{align} 
\label{certificate-G-v2}
\begin{split}
H(x)&=
(x+1)H_0(\x)-xH_1(\x)+(\x+1)^2\big(x+\frac{1}{b-a+1}\big) H_2(\x)\\
&\hspace{4cm}
-(\x+1)\big(x+\frac{1}{b-a+1}\big)xH_3(\x),
\end{split}
\end{align}
where the degree of each summand is at most $\deg H$.
Note that $n=2\lfloor n/2\rfloor+(n\;\text{mod}\; 2)$. 
Replacing $H$ in the equality \eqref{express-F-with-G}
with the right-hand side of \eqref{certificate-G-v2}, we get
\begin{align}
\label{F-in-terms-of-G-v2}
\begin{split}
    F(x)&=
    (x-a+1)^{n\;\text{mod}\; 2+1}
    \Big((x-a)K_0(x)+K_1(x)+(x-a)^2(x-b)K_2(x)\\
    &\hspace{5cm}+(x-a)(x-b)K_3(x)\Big),
\end{split}
\end{align}
where 
\begin{align*} 
    K_i(x)&=(x-a+1)^{2\lfloor n/2\rfloor-2}\;H_i\big(-\frac{1}{x-a+1}\big)
    \quad\text{for }i=0,1,
    \\
    K_i(x)&=\frac{1}{b-a+1}(x-a+1)^{2\lfloor n/2\rfloor-4}\;H_i\big(-\frac{1}{x-a+1}\big)
    \quad\text{for }i=2,3,
\end{align*}
are matrix polynomials, positive semidefinite on $\RR$.
Indeed, \eqref{F-in-terms-of-G-v2} follows from the fact that for
\begin{align*}
    f_3(x)&:=x+1,\quad 
    f_4(x):=-x, \quad
    f_5(x):=(x+1)^2\big(x+\frac{1}{b-a+1}\big),\\
    f_6(x)&:=-(x+1)\big(x+\frac{1}{b-a+1}\big)x,
\end{align*}
we have
\begin{align*} 
(x-a+1)f_3\big(-\frac{1}{x-a+1}\big)&=x-a,\\
(x-a+1)f_4\big(-\frac{1}{x-a+1}\big)&=1,\\
(x-a+1)^3f_5\big(-\frac{1}{x-a+1}\big)&=\frac{(x-a)^2(x-b)}{b-a+1},\\
(x-a+1)^3f_6\big(-\frac{1}{x-a+1}\big)&=\frac{(x-a)(x-b)}{b-a+1}.
\end{align*}
Using \eqref{F-in-terms-of-G-v2}, it follows that 
\begin{equation}
\label{F-with-G-v4}
F(\x)=
        G_0(\x)+(\x-a)G_1(\x)+(\x-a)(\x-b)G_2(\x)+
        (\x-a)^2(\x-b)G_3(\x),
\end{equation}
where
\begin{align*}
G_0(x)
&:=
\left\{
    \begin{array}{rl}
        (x-a)^2K_0(x)+K_1(x),&   \text{if }n\;\text{mod}\;2=0,\\
        (x-a)^2K_1(x),&   \text{if }n\;\text{mod}\;2=1,
    \end{array}
\right.\\
G_1(x)
&:=
\left\{
    \begin{array}{rl}
        K_0(x)+K_1(x),&   \text{if }n\;\text{mod}\;2=0,\\
        (x-a)^2K_0(x),&   \text{if }n\;\text{mod}\;2=1,
    \end{array}
\right.\\
G_2(x)
&:=
\left\{
    \begin{array}{rl}
        (x-a)^2K_2(x)+K_3(x),&   \text{if }n\;\text{mod}\;2=0,\\
        (x-a)^2K_3(x),&   \text{if }n\;\text{mod}\;2=1,
    \end{array}
\right.\\
G_3(x)
&:=
\left\{
    \begin{array}{rl}
        K_2(x)+K_3(x),&   \text{if }n\;\text{mod}\;2=0,\\
        (x-a)^2K_2(x),&   \text{if }n\;\text{mod}\;2=1.
    \end{array}
\right.
\end{align*}
Note that the degree of each summand on the right-hand side of \eqref{F-with-G-v4} is bounded 
by $\deg F$ and each $G_i(x)$ is positive semidefinite on $\RR$, whence by Theorem \ref{thm:Fejer-Riesz} it is a sum of at most two symmetric squares. This proves Theorem \ref{thm:unbounded-interval} for the case $\deg H$ is odd.
\qed
\bigskip

\subsection{Proof of Theorem \ref{thm:two-intervals}}
Applying an affine linear transformation we can assume without loss of generality that $a=-1$, $b=1$ and $c>1$.
Let $F$ be a matrix polynomial, positive semidefinite on
$K=(-\infty,-1]\cup\{1\}\cup [c,\infty)$. Note that $\deg F$ is even and write $2n:=\deg F$, $n\in \NN\cup \{0\}$.
Define $\phi(x)=-\frac{1}{x}$ and note that 
    $\overline{\phi(K)}=\{-1\}\cup[-\frac{1}{c},1]$.
Then the polynomial 
    $$H(x):=x^{2n}F\big(-\frac{1}{x}\big)$$ 
is positive semidefinite on
$\overline{\phi(K)}$.
Note that
\begin{equation} 
\label{express-F-with-G-new}
    F(x)=x^{2n}H\big(-\frac{1}{x}\big).
\end{equation}
Writing $F(x)=\sum_{i=0}^{2n}F_ix^{i}$, we have 
$H(x)=\sum_{i=0}^{2n} F_i(-x)^{2n-i}.$
We separate two cases according to the parity of $\deg H$.\\

\noindent \textbf{Case 1:} $\deg H$ is even. 
By Theorem \ref{thm:bounded-interval}, there are matrix polynomials $H_0(x)$, $H_1(x)$, $H_2(x)$, each a sum of at most two symmetric squares, such that
\begin{equation} 
\label{certificate-G-new}
H(x)=
    H_0(\x)
    +(\x+1)\big(\x+\frac{1}{c}\big)H_1(\x)
    +(\x+1)(1-x)H_2(\x),
\end{equation}
where the degree of each summand is at most $\deg H$.
Replacing $H$ in the equality \eqref{express-F-with-G-new}
with the right-hand side of \eqref{certificate-G-new}, we get
\begin{equation}
\label{F-in-terms-of-G-new}
    F(x)=G_0(x)+(x-c)(x-1)G_1(x)+(x+1)(x-1)G_2(x),
\end{equation}
where 
\begin{align*} 
    G_0(x)&=x^{2n}H_0\big(-\frac{1}{x}\big),\quad
    G_1(x)=\frac{1}{c}x^{2n-2}\;H_1\big(-\frac{1}{x}\big),\quad
    G_2(x)=x^{2n-2}\;H_2\big(-\frac{1}{x}\big)
\end{align*}
are matrix polynomials, positive semidefinite on $\RR$.
Note that the degree of each summand on the right-hand side of \eqref{F-in-terms-of-G-new} is bounded 
by $2n$ and each $G_i(x)$ is positive semidefinite on $\RR$, whence by Theorem \ref{thm:Fejer-Riesz} it is a sum of at most two symmetric squares. This proves Theorem \ref{thm:two-intervals} for the case $\deg H$ is even.\\

\noindent \textbf{Case 2:} $\deg H$ is odd. 
By Theorem \ref{thm:bounded-interval}, there are matrix polynomials $H_0(x)$, $H_1(x)$, $H_2(x)$, $H_3(x)$, each a sum of at most two symmetric squares, such that
\begin{align} 
\label{certificate-G-v2-new}
\begin{split}
H(x)&=
(x+1)H_0(\x)+(1-x)H_1(\x)+(\x+1)^2\big(x+\frac{1}{c}\big) H_2(\x)\\
&\hspace{4cm}
+(\x+1)\big(x+\frac{1}{c}\big)(1-x)H_3(\x),
\end{split}
\end{align}
where the degree of each summand is at most $\deg H$. 
Replacing $H$ in the equality \eqref{express-F-with-G-new}
with the right-hand side of \eqref{certificate-G-v2-new}, we get
\begin{align}
\label{F-in-terms-of-G-v2-new}
\begin{split}
    F(x)&=x\Big((x-1)K_0(x)+(x+1)K_1(x)+(x-1)^2(x-c)K_2(x)\\
    &\hspace{5cm}+(x-1)(x+1)(x-c)K_3(x)\Big),
\end{split}
\end{align}
where 
\begin{align*} 
    K_i(x)&=x^{2n-2}\;H_i\big(-\frac{1}{x}\big)
    \quad\text{for }i=0,1,
    \\
    K_i(x)&=\frac{1}{c}x^{2n-4}\;H_i\big(-\frac{1}{x}\big)
    \quad\text{for }i=2,3,
\end{align*}
are matrix polynomials, positive semidefinite on $\RR$.
Using \eqref{F-in-terms-of-G-v2-new} and $x=\frac{x-1}{2}+\frac{x+1}{2}$, it follows that 
\begin{equation}
\label{F-with-G-v4-new}
F(\x)=
        G_0(\x)+(\x+1)(x-1)G_1(\x)+(\x-1)(\x-c)G_2(\x)+
        (x+1)(\x-1)^2(\x-c)G_3(\x),
\end{equation}
where
\begin{align*}
G_0(x)
&:=\frac{(x-1)^2}{2}K_0(x)+\frac{(x+1)^2}{2}K_1(x),\\
G_1(x)
&:=\frac{1}{2}K_0(x)+\frac{1}{2}K_1(x)\\
G_2(x)
&:=\frac{(x-1)^2}{2}K_2(x)+\frac{(x+1)^2}{2}K_3(x)\\
G_3(x)
&:=\frac{1}{2}K_2(x)+\frac{1}{2}K_3(x).
\end{align*}
Note that the degree of each summand on the right-hand side of \eqref{F-with-G-v4-new} is bounded 
by $\deg F$ and each $G_i(x)$ is positive semidefinite on $\RR$, whence by Theorem \ref{thm:Fejer-Riesz} it is a sum of at most two symmetric squares. This proves Theorem \ref{thm:two-intervals} for the case $\deg H$ is odd.
\qed

}

\bigskip\bigskip

\CB{
 \noindent \textbf{Acknowledgement.}\
We would like to thank the anonymous referee for carefully reading our manuscript and
many suggestions to improve the overall presentation of our results.}


\end{document}